\documentclass[9pt]{amsart}
\usepackage{amssymb}
\usepackage{color}
\newtheorem{theorem}{Theorem}

\newtheorem{lemma}[theorem]{Lemma}

\theoremstyle{definition}
\newtheorem{definition}[theorem]{Definition}
\theoremstyle{remark}

\newtheorem{remark}[theorem]{Remark}

\numberwithin{equation}{section}

\begin{document}

\title[Bloch type spaces]
{Bloch type spaces on the unit ball of a Hilbert space}

\author[Z.  Xu]{Zhenghua Xu }
\thanks{This work was supported by the NNSF  of China (11371337, 11771412).}
\address{Zhenghua Xu, School of Mathematics, HeFei University of Technology, No. 420  Feicui Road, Hefei 230601, Anhui, China}
\email{zhxu$\symbol{64}$hfut.edu.cn}


\begin{abstract}
In this article, we   initiate the study of   Bloch type spaces on the unit ball of a Hilbert space. As applications, the Hardy-Littlewood theorem in    infinite-dimensional Hilbert spaces and  characterizations of some holomorphic function spaces related to the Bloch type space  are presented.
\end{abstract}
\maketitle
\section{Introduction}
The classical Bloch space  of holomorphic functions on the  unit disk $\mathbb{D}$ of  the complex plane $\mathbb{C}$  was  extended to the higher dimension  cases. In 1975, K. T. Hahn introduced  the notion of Bloch functions on bounded homogeneous domains in $\mathbb{C}^{n}$ using terminology   from differential geometry \cite{Hahn}. In \cite{Timoney1,Timoney2},  R. M.  Timoney  studied further Bloch functions  on bounded homogeneous domains in terms of the Bergman metric. In \cite{Krantz-Ma},  S. G. Krantz and  D. Ma considered  function theoretic and functional analytic properties  of Bloch functions on strongly pseudoconvex domain. To have a more complete insight on the theory of the Bloch space in the  finite dimensional  space, see  the book  \cite{Zhu} by Zhu.

Recently,  Bloch functions on the unit ball   of an infinite-dimensional complex Hilbert space  have been  studied by Blasco, Galindo   and Miralles \cite{Blasco}. In this article, we shall continue the study in \cite{Blasco} and consider Bloch type  spaces on the unit ball of a Hilbert space.  Especially, we  give four semi-norms of the Bloch type space  and   show their equivalences and present some other equivalent characterizations for Bloch functions from the  geometric perspective which are  the  infinite-dimensional   generalization of \cite[Theorem 3.4]{Timoney1}.

Let $\mathbb{B}$ be the open  unit ball of the complex Hilbert space $E$ and $\partial \mathbb{B}$ be the unit sphere. The
class of all holomorphic functions   $f:\mathbb{B} \rightarrow \mathbb{C}$  is denoted by $H(\mathbb{B})$.  Denote by Aut$(\mathbb{B})$ the group of all biholomorphic mappings of $\mathbb{B}$
onto itself.   For $0\leq\alpha <\infty$, let $H_{\alpha}^{\infty}$ be the space of holomorphic functions $f\in H(\mathbb{B})$ satisfying
  $$\sup_{x\in \mathbb{B}}(1-\|x\|^{2})^{\alpha} |f(x)|<\infty.$$
We abbreviate  $H^{\infty}=H_{0}^{\infty}$ for $ \alpha=0$.

 The classical $\alpha$-Bloch space   is the space of holomorphic  functions $F: \mathbb{D}  \rightarrow \mathbb{C}$ satisfying
$$\|F\|_{\mathcal{B}^{\alpha}(\mathbb{D})}:=\sup_{ z\in \mathbb{D}}( 1- |z|^{2})^{\alpha} |F'(z)| <+\infty.$$

Now we introduce four  semi-norms of  the Bloch type space  for $f\in H(\mathbb{B})$.

Denote
$$\|f\|_{1,\alpha}:=\sup_{x\in\mathbb{B}}  ( 1-\|x\|^{2})^{\alpha}\|D f(x)\| ,$$
$$\|f\|_{2,\alpha}:=\sup_{x\in\mathbb{B}}   ( 1- \|x\|^{2})^{\alpha}|\mathcal{R} f(x)|,$$
$$\|f\|_{3,\alpha}:=\sup_{y\in  \partial \mathbb{B}}    \|f_{y}\|_{\mathcal{B}^{\alpha}(\mathbb{D})},$$
where  $\mathcal{R} f(x)=D f(x)(x)$, $f_{y}(z)=f(zy)$ for $z\in \mathbb{D}$.

The   M\"{o}bius transforms of $\mathbb{B}$ are holomorphic mappings $\varphi_{a}, a\in \mathbb{B}$,   given by
 $$\varphi_{a}(x)=(P_{a}+s_{a}Q_{a})(m_{a}(x)),$$
 where $s_{a}=\sqrt{1-\|a\|^{2}}$, $P_{a}(x)= \frac{\langle x,a\rangle}{\langle a,a\rangle}a, Q_{a}=Id-P_{a}$
and $m_{a}(x)=\frac{a-x}{1-\langle x,a\rangle}.$

Define
$$\|f\|_{4,\alpha}:=\sup_{x\in \mathbb{B}}(1-\|x\|^{2})^{\alpha-1}\| \widetilde{\nabla} f(x)\|,$$
where $\widetilde{\nabla} f(x)=D  f \circ \varphi_{x}(0)$ with $\varphi_{x}\in$ Aut$(\mathbb{B})$.

Note that, by \cite[Lemma 3.5]{Blasco},
$$\| \widetilde{\nabla} f(x)\|=\sup_{w\neq0} \frac{( 1- \|x\|^{2})|D f(x)(w)|}{\sqrt{( 1- \|x\|^{2})\|w\|^{2}+|\langle w,x\rangle|^{2}}}.$$
Hence,  we have
$$\|f\|_{4,\alpha} =\sup_{x\in \mathbb{B}}\sup_{w\neq0} \frac{( 1- \|x\|^{2})^{\alpha}|D f(x)(w)|}{\sqrt{( 1- \|x\|^{2})\|w\|^{2}+|\langle w,x\rangle|^{2}}}.$$

For $\alpha>0$, denote
$${\mathcal{B}^{\alpha}}=\big\{ f\in H(\mathbb{B}):  \|f\|_{1,\alpha} <+\infty \big\},$$
and
$$ T_{\alpha}=\big\{ f\in H(\mathbb{B}):  \|f\|_{4,\alpha} <+\infty \big\}.$$

Equipped  with the norm 	$\|f\|_{\alpha}=
|f(0)|+\|f\|_{1,\alpha}$ for $f\in {\mathcal{B}^{\alpha}}$, the Bolch type space ${\mathcal{B}^{\alpha}}$ becomes a Banach space as usual.
Denote  the class of Bloch functions defined on $\mathbb{B}$ by  $\mathcal{B}$  instead of $\mathcal{B}^{1}$.  The little Bloch space will be denoted by $\mathcal{B}_{0}$ which  consists of functions $f\in \mathcal{B}$ such that
$$\lim_{\|x\|\rightarrow 1^{-}} \| \widetilde{\nabla} f(x)\|=0.$$

Now our  main results can be described  as follows.
\begin{theorem}\label{main-first}Let $\alpha>0$,  $\mathbb{B}$ be the open unit ball of the complex Hilbert space $E$.  If  $f$ is a complex-valued holomorphic function on $\mathbb{B}$, then the three semi-norms $\|f\|_{1,\alpha}$, $\|f\|_{2,\alpha}$ and $\|f\|_{3,\alpha}$ are equivalent.
\end{theorem}
\begin{theorem}\label{main-first1}
Let   $\mathbb{B}$ be the open unit ball of the complex Hilbert space $E$ with dim$E\geq2$ and  $f$ be a complex-valued holomorphic function
 on $\mathbb{B}$.
\begin{enumerate}
\item [(i)]  If $0<\alpha<1/2$, then   $f\in T_{\alpha}$ iff $f$ is constant.
\item [(ii)]   If  $\alpha=1/2$ , then $f\in T_{1/2}$ iff
$|D f(x)(y)| $  is bounded for all $x\in \mathbb{B}$ and $y\in \partial \mathbb{B}$ with $\langle x,y\rangle=0$.
\item [(iii)]  If $\alpha>1/2$, then  two seminorms $\|f\|_{1,\alpha}$ and $\|f\|_{4,\alpha}$   are equivalent.
\end{enumerate}

\end{theorem}

Note that the results for $\alpha=1$ in Theorems \ref{main-first} and
\ref{main-first1} have been obtained in \cite{Blasco} and  the
 condition dim$E\geq2$ in    Theorem  \ref{main-first1}
(i) and (ii)  can not be deleted in general.

From Theorem \ref{main-first} and its proof, the interested reader   can give some equivalent characterizations for  the little Bloch type space $\mathcal{B}_{0}$.   It is worth mentioning that the approach for the finite-dimensional case  depends   usually on the  integral representation  for holomorphic functions. However, it may fail for the infinite-dimensional  setting. Hence  we need to  overcome the restriction of dimension in achieving    our main results.

In this paper, we add some more equivalent characterizations for Bloch functions. To this end, we now give the  definition of a schlicht disk for the holomorphic  function defined in Hilbert spaces as in the case of several complex variables.

\begin{definition}Let   $f$ be a  holomorphic  function on a domain $\Omega$ in $E$.  For $z_{0}\in \mathbb{C}$, the disk
$$D(z_{0},r)= \{z \in \mathbb{C} : |z-z_0|<r\} $$
is called a schlicht disk in the range of $f$ if
there exists a  holomorphic  mapping   $g:\mathbb{D} \rightarrow \Omega $ such  that $ f\circ g(z)=z_0+rz$.
\end{definition}

Following the paper \cite{Lehto},  a   holomorphic function  $f:\mathbb{B}\rightarrow \mathbb{C}$ is said to be  normal  if
$$ M_{f}=\sup\Big\{\frac{(1-\|x\|^{2})\|Df(x)\|}{1+|f(x)|^{2}}:x\in \mathbb{B}\Big\} <+\infty.$$

With these two concepts in hand,   our second  main result can be established  as follows.
\begin{theorem}\label{second}
Let $f\in H(\mathbb{B})$.  Then the following conditions  are equivalent:\begin{enumerate}
\item [(i)] $f$ is a Bloch function;
\item [(ii)] the radii of schlicht disks in the range of $f$ are bounded above;
  \item[(iii)]
  the family $$\big\{ f\circ g :g\in H(\mathbb{D},\mathbb{B}) \big\}$$
  is a family of Bloch functions with uniformly bounded Bloch norm;
 \item [(iv)]
  the family $$\big\{ f\circ g- f\circ g(0) :g\in H(\mathbb{D},\mathbb{B})\big \}$$ is a normal family in the sense of Montel;
 \item   [(v)]
     the family
 $$  \mathcal{F}_{f}:=\{h=f\circ\phi-f\circ\phi(0) :   \phi\in \text{Aut}(\mathbb{B})\} $$
is a family of  normal functions such that  $M_{h}$ is uniformly bounded.

 \end{enumerate}
\end{theorem}

\begin{remark}
The Bloch space  on bounded symmetric domains in  arbitrary complex  Banach spaces was considered  in \cite{Deng}.
Although we treat only  Bloch functions defined   on the unit ball of a Hilbert space, all results in Theorem \ref{second} can be proved  in the more general setting of bounded symmetric domains.
\end{remark}

The remaining part of this paper is organized as follows.   Theorems \ref{main-first}, \ref{main-first1}  and \ref{second} are proved in Section $2$. In  Section $3$,  the Hardy-Littlewood theorem in infinite-dimensional Hilbert spaces is established as an application  of  Theorem  \ref{main-first}. In addition, we give some equivalent characterizations  for holomorphic function spaces related to the Bloch type space which are the  generalizations of main results in \cite{Dai,Zhao} in  infinite-dimensional Hilbert  spaces.

\section{Proof of   Theorems \ref{main-first}, \ref{main-first1} and \ref{second}}
Throughout this paper, denote by $C$ an absolute positive constant and
 by $C(\alpha)$   a positive constant depending on $\alpha$ only. They may have different values at different places.

In order to prove Theorems \ref{main-first} and \ref{main-first1}, we
establish two lemmas.
 \begin{lemma}\label{Ouyang1}
 Let $\alpha\geq0$ and  $f:\mathbb{B} \rightarrow  \mathbb{C} $  be a holomorphic function.
 \begin{enumerate}
\item [(i)] If $|f(x)|\leq ( 1- \|x\|^{2})^{-\alpha} $ for all $ x\in \mathbb{B}$, then
 $$|Df(x)(y)|\leq C(\alpha) (1-\|x\|^{2})^{-\alpha-\frac{1}{2}},$$
for all  $x\in \mathbb{B}, y\in \partial \mathbb{B}$ with $\langle x,y\rangle=0$.
\item [(ii)] If $|Df(x)(y)|\leq ( 1- \|x\|^{2})^{-\alpha}$ for all  $x\in \mathbb{B}, y\in \partial \mathbb{B}$ with $\langle x,y\rangle=0$, then
 $$|\mathcal{R}f(x)|\leq C(\alpha) (1-\|x\|^{2})^{-\alpha-\frac{1}{2}}, \ \ \forall x\in \mathbb{B}.$$
\item [(iii)]
 If $|f(x)|\leq ( 1- \|x\|^{2})^{-\alpha}$ for all  $x\in \mathbb{B}$, then
 $$|\mathcal{R}f(x)|\leq C(\alpha) (1-\|x\|^{2})^{-\alpha-1}, \ \ \forall x\in \mathbb{B}.$$
\end{enumerate}
 \end{lemma}
\begin{proof}
(i) Let  $x=rx'\in \mathbb{B}, y\in \partial \mathbb{B}$ be such that  $x'\in \partial \mathbb{B}$ with $\langle x,y\rangle=0$.
 Let us consider the holomorphic function $ F: B^{2}\rightarrow \mathbb{C}$ given by
$ F(z_{1},z_{2})=f(z_{1}x'+z_{2}y)$, where $B^{2}$ is the open unit ball of $\mathbb{C}^{2}$. By assumption, we get  $|F(z_{1},z_{2})|\leq ( 1- |z_{1} |^{2}-|z_{2} |^{2})^{-\alpha}$.  By \cite[Lemma 6.4.6]{Rudin}, it holds that
$$ \Big| \frac{\partial F}{\partial z_2}(r,0)\Big|\leq C(\alpha)( 1- r^{2})^{-\alpha-\frac{1}{2}}, $$
that is
$$|Df(x)(y)|\leq C(\alpha) (1-\|x\|^{2})^{-\alpha-\frac{1}{2}}.$$

(ii) Based on the  result for $\mathbb{C}^{2}$ (cf. \cite[Lemma 1(a)]{Yang}), we can obtain the desired estimate  applying  the same method as in (i). (iii) is just a corollary from (i) and (ii).
\end{proof}

\begin{lemma}\label{Ouyang2}
Let $\alpha\geq0$ and  $f:\mathbb{B} \rightarrow  \mathbb{C} $  be a holomorphic function.    If
  $$|D\mathcal{R}f(x)(y)|\leq   (1-\|x\|^{2})^{-\alpha-\frac{1}{2}},$$
   for all $x\in \mathbb{B}$ and $y\in \partial \mathbb{B}$ with $\langle x,y\rangle=0$,
 then $$|Df(x)(y)|\leq  C(\alpha) (1-\|x\|^{2})^{-\alpha},$$
 for all $x\in \mathbb{B}$ and $y\in \partial \mathbb{B}$ with $\langle x,y\rangle=0$.
\end{lemma}
 \begin{proof}
 For  $f\in H(\mathbb{B})$, we   rewrite it as $\sum_{n=0}^{\infty} P_{n}(x)$, where $P_{n}$ is an $n$-homogeneous polynomial, that is, the restriction to the diagonal of a continuous $n$-linear form on the $n$-fold space $E\times \cdots\times E$. Then $ \mathcal{R} P_{n}=nP_{n}$ and $DP_{n}$ is $(n-1)$-homogeneous $ (n\geq1)$, so that
 $$(D\mathcal{R}P_{n})(tx')= n(D P_{n})(tx')=nt^{n-1}(D P_{n})(x'),$$
 where $ x' \in \partial B_{E}, 0\leq t<1$.

Hence,
 $$\int_{0}^{r}(D\mathcal{R}P_{n})(tx')dt=r^{n}(D P_{n})(x')=r(D P_{n})(rx'),$$
 which leads to
 $$r D f(rx')(y)=  \int_{0}^{r} (D\mathcal{R}f)(tx')(y) dt.$$
It follows that, by assumption, for $r\in [1/2,1)$, $y\in \partial \mathbb{B}$ with $\langle x',y\rangle=0$,
 $$|D f(rx')(y)|\leq 2\int_{0}^{r}|(D\mathcal{R}f)(tx')(y)|dt
 \leq  2\int_{0}^{r}(1-t^{2})^{-\alpha-\frac{1}{2}}dt,$$
 then
  $$(1-\|x\|^{2})^{\alpha}|D f(x)(y)|
 \leq  2\int_{0}^{r}(1-t)^{-\frac{1}{2}}dt=4(1-\sqrt{1-r})<4,$$
 for $\|x\| \in [1/2,1)$, $y\in \partial \mathbb{B}$ with $\langle x,y\rangle=0$.

 Hence, by the maximum principle for holomorphic functions,
 $$|D f(x)(y)|\leq 4 (\frac{4}{3})^{\alpha},$$
  for $\|x\| \in  [0,1/2]$, $y\in \partial \mathbb{B}$ with $\langle x,y\rangle=0$,
 as desired.
 \end{proof}

We now   are in a position to prove Theorems \ref{main-first} and \ref{main-first1}.
\begin{proof}[Proof of Theorem \ref{main-first}]
Let $x\in \mathbb{B}$ be fixed. For  any  $y\in \partial \mathbb{B}$, by the projection theorem, we can write $y=z_{1}x+z_{2}x_{1}$ for  some  $z_{1}, z_{2}\in \mathbb{C}, x_{1}\in\partial \mathbb{B}$ with  $\langle x,  x_{1}\rangle=0$.

Note that
$$|z_{1}|^{2}\|x\|^{2}+|z_{2}|^{2}=\|y\|^{2}=1.$$
Hence, for $1/2\leq\|x\|<1$, we have
\begin{eqnarray}\label{condition-1}
|Df(x)(y)| \leq 2|\mathcal{R}f(x)|+|Df(x)(x_{1})|.
\end{eqnarray}

Suppose that $\|f\|_{2,\alpha}=1$, by Lemma \ref{Ouyang1} (i), then we have
$$|D\mathcal{R}f(x)(x_{1})| \leq  C(\alpha) (1-\|x\|^{2})^{-\alpha-\frac{1}{2}},$$
so by Lemma  \ref{Ouyang2}
\begin{eqnarray}\label{condition-2}
 |Df(x)(x_{1})| \leq  C(\alpha) (1-\|x\|^{2})^{-\alpha}.
\end{eqnarray}
Hence, by (\ref{condition-1}) and (\ref{condition-2}), for $1/2\leq\|x\|<1$,
$$\|Df(x)\|=\sup_{y\in \partial \mathbb{B}}|Df(x)(y)|  \leq  (2+C(\alpha)) (1-\|x\|^{2})^{-\alpha}. $$
By the maximum principle for holomorphic mappings in Banach spaces,
we have
 $$ \|f\|_{1,\alpha}\leq  2+C(\alpha). $$
Consequently,
\begin{eqnarray}\label{condition-3}
  \|f\|_{1,\alpha}\leq C(\alpha) \|f\|_{2,\alpha}.
\end{eqnarray}
It is clear that
 \begin{eqnarray}\label{condition-4}
 \|f\|_{2,\alpha}\leq \|f\|_{1,\alpha}.
\end{eqnarray}
Hence the two seminorms $\|\cdot\|_{1,\alpha}$  and $\|\cdot\|_{2,\alpha}$ are equivalent by  (\ref{condition-3}) and  (\ref{condition-4}).

Notice that $zf_{y}'(z) = \mathcal{R}f(zy)$ for any holomorphic $f$ defined on $\mathbb{B}$,  then we have that
 \begin{eqnarray}\label{condition-5}
 \|f\|_{2,\alpha} \leq \| f\|_{3,\alpha}.
\end{eqnarray}
For $1/2\leq |z|<1$, we have
$$(1-|z|^{2})^{\alpha}|f_{y}'(z)| \leq 2(1-|z|^{2})^{\alpha}|\mathcal{R}f(zy)|.$$
Hence
$$\sup_{|z|\geq 1/2 }(1-|z|^{2})^{\alpha}|f_{y}'(z)|
\leq 2\sup_{\|x\|\geq 1/2 }(1-\|x\|^{2})^{\alpha}|\mathcal{R}f(x)|
\leq 2\|f\|_{2,\alpha}.$$
Combining  this with the    maximum principle for holomorphic functions, one can show that
 \begin{eqnarray}\label{condition-6}
  \|f\|_{3,\alpha} \leq 2(\frac{4}{3})^{\alpha}\|f\|_{2,\alpha}.
\end{eqnarray}

 Hence, by   (\ref{condition-5}) and  (\ref{condition-6}), the two semi-norms $\| \cdot\|_{2,\alpha}$  and $\|\cdot\|_{3,\alpha}$ are equivalent.
\end{proof}
\begin{remark}
From the proof of Theorem \ref{main-first}, we see that   semi-norms $\| f\|_{2,\alpha}$  and $\|f\|_{3,\alpha}$ are equivalent for any  holomorphic function $f$ defined on the unit ball of  Banach spaces.
\end{remark}

\begin{proof}[Proof of Theorem \ref{main-first1}]
(i) Let  $0<\alpha< 1/2 $ and   $f\in T_{\alpha}$. It holds that
$$|Df(x)(y)|\leq \|f\|_{4,\alpha} ( 1- \|x\|^{2})^{\frac{1}{2}-\alpha},$$
for all $x\in \mathbb{B}$ and $y\in \partial \mathbb{B}$ with $\langle x,y\rangle=0$.

Now let us show first that
\begin{eqnarray}\label{Holo}
 g(x):=D f(x)(y) \equiv 0,\quad   x\in \mathbb{B}\  {\rm{and}} \ y\in E  \ {\rm{with}} \ \langle x,y\rangle=0.
\end{eqnarray}

To this end, for fixed $x\in \partial \mathbb{B}$ and $y\in E$ such that  $\langle x,y\rangle=0$, we consider  the slice function $h(z)=g(zx)$ on $\mathbb{D}$ satisfying  the property
$$ \limsup_{z \rightarrow \partial \mathbb{D}} |h(z)|=0.$$
 Applying  the maximum principle to the holomorphic mapping  $h$,  we see that $h(z)\equiv0$ on  $\mathbb{D}$, then (\ref{Holo}) follows.

Let $x\in \mathbb{B}$ be fixed. For every $w\in \partial \mathbb{B}$,   by the projection theorem, we can write $w=zx+y$ with   $\langle x,  y\rangle=0$ for some $z\in \mathbb{C}, y\in \mathbb{B}$. Hence, it follows that
\begin{eqnarray}\label{-}
 |z| \|x\|\leq \|w\| =1.
\end{eqnarray}
Now we have, by  (\ref{Holo}),
$$|D f(x)(w) | \leq | Df(x)(zx)|+| Df(x)(y)|=   |z|  | \mathcal{R} f(x)|.$$
Combining this with  (\ref{-}), we obtain
$$  \|Df(x)\| =\sup_{w\in \partial \mathbb{B}} |D f(x)(w) |\leq \frac{1}{\|x\|} | \mathcal{R} f(x)| \leq \|D f(x)\|, $$
which forces that  $Df(x)$ and $\overline{x}$ are complex linear.

Note that $D f: \mathbb{B}\rightarrow E^{\ast}$ is  holomorphic and thus $D f(x)\equiv0 $ on $\mathbb{B}$. Hence,   $f$ is constant,  as desired.

(ii) Suppose that $|D f(x)(y)|\leq C$    for all $x\in \mathbb{B}$ and $y\in \partial \mathbb{B}$ with $\langle x,y\rangle=0$. Let us show that
$f\in T_{1/2}$.

By  Lemma \ref{Ouyang1} (ii), it follows that
\begin{eqnarray}\label{Rf}
 |\mathcal{R}f(x)|\leq C( 1- \|x\|^{2})^{-\frac{1}{2}}, \quad \forall x\in \mathbb{B}.
\end{eqnarray}
For fixed  $x\in \mathbb{B}$, every $w\in E$ can be decomposed as  $w=zx+y$ with $\langle x,y\rangle = 0$ for some $z\in \mathbb{C}$, then
 $$( 1- \|x\|^{2})\|w\|^{2}+|\langle w,x\rangle|^{2} = |z|^{2}\|x\|^{2}+( 1- \|x\|^{2})\|y\|^{2}.$$
It follows that
\begin{eqnarray*}
 && \frac{|Df(x)(w)|( 1- \|x\|^{2})^{1/2}}{\sqrt{( 1- \|x\|^{2})\|w\|^{2}+|\langle w,x\rangle|^{2}}}
 \\
 &\leq& \frac{|z| \|\mathcal{R} f(x)\|( 1- \|x\|^{2})^{1/2}}{ \sqrt{|z|^{2}\|x\|^{2}+( 1- \|x\|^{2})\|y\|^{2}}} + \frac{|D f(x)(y)|( 1- \|x\|^{2})^{1/2}}{ \sqrt{|z|^{2}\|x\|^{2}+( 1- \|x\|^{2})\|y\|^{2}}}
 \\
 &\leq& \frac{1}{\|x\|}|\mathcal{R}f(x)|( 1- \|x\|^{2})^{1/2}   + |D f(x)(\frac{y}{\|y\|})|.
 \end{eqnarray*}
Combining  this with (\ref{Rf}), we obtain that, for $1/2\leq\|x\|<1$,
\begin{eqnarray}\label{Rf1}
\frac{|Df(x)(w)|( 1- \|x\|^{2})^{1/2}}{\sqrt{( 1- \|x\|^{2})\|w\|^{2}+|\langle w,x\rangle|^{2}}}\leq C.
\end{eqnarray}

Furthermore, $\|Df(x)\| $ is bounded  for $\|x\|< 1/2$   by the maximum principle for  holomorphic mappings. Then
\begin{eqnarray}\label{Rf2}
 \frac{|Df(x)(w)|( 1- \|x\|^{2})^{1/2}}{\sqrt{( 1- \|x\|^{2})\|w\|^{2}+|\langle w,x\rangle|^{2}}}\leq
 \|Df(x)\| \leq C.
\end{eqnarray}
 Hence, by (\ref{Rf1}) and (\ref{Rf2}),  $f\in T_{1/2}$.

Conversely,  suppose   $f\in T_{1/2}$, then we have that,  for $x\in \mathbb{B}$ and $y\in \partial \mathbb{B}$ with $\langle x,y\rangle=0$,
$$ |D f(x)(y)|=\frac{|Df(x)(y)|( 1- \|x\|^{2})^{1/2}}{\sqrt{( 1- \|x\|^{2})\|y\|^{2}+|\langle y,x\rangle|^{2}}}.
$$
Hence, $|D f(x)(y)|$ is bounded.

(iii) Note that  $T_{\alpha} \subseteq {\mathcal{B}^{\alpha}}$. It remains to   show  ${\mathcal{B}^{\alpha}} \subseteq T_{\alpha}$ for
$\alpha>\frac{1}{2}$.
Let us  show first  that if the holomorphic function  $F:B^{2}\rightarrow \mathbb{C}  $ satisfies
 $$\|F\|_{\mathcal{B}^{\alpha}(B^{2})}=\sup_{z=(z_{1},z_{2})\in B^{2} }(1- |z |^{2})^{ \alpha}\|DF(z)\| <+\infty,$$
then
\begin{eqnarray}\label{TM}
 \big|\frac{\partial F}{\partial z_{2}}(z_{1},0)\big|(1-  |z_{1} |^{2})^{ \alpha-\frac{1}{2}} \leq  C(\alpha) \|F\|_{\mathcal{B}^{\alpha}(B^{2})}, \quad \forall  z_{1}\in \mathbb{D}.
\end{eqnarray}
To this end, set $s=\frac{1}{\sqrt{3}}(1-|z_{1}|^{2})^{\frac{1}{2}} $ for  $z_{1}\in \mathbb{D}$. It holds that
$$\frac{\partial ^{2} F}{\partial z_{2} \partial z_{1}}(z_{1},0)
=\frac{1}{2\pi i} \int_{|w|=s } \frac{\partial   F}{\partial z_{1}} (z_{1},w) \frac{d w}{w^{2}}, $$
then
$$\Big|\frac{\partial ^{2} F}{\partial z_{2} \partial z_{1}}(z_{1},0)\Big|\leq \frac{\|F\|_{\mathcal{B}^{\alpha}(B^{2})}}{s(1-|z_{1}|^{2}-s^{2})^{\alpha}}
=\frac{3^{\alpha+\frac{1}{2}}\|F\|_{\mathcal{B}^{\alpha}(B^{2})}}
{2^{\alpha}(1-|z_{1}|^{2})^{\alpha+\frac{1}{2}}}\leq\frac{3^{\alpha
+\frac{1}{2}}\|F\|_{\mathcal{B}^{\alpha}(B^{2})}}
{2^{\alpha}(1-|z_{1}|)^{\alpha+\frac{1}{2}}}. $$
Combining  with the formula
$$\frac{\partial   F}{\partial z_{2}  }(z_{1},0)
-\frac{\partial   F}{\partial z_{2}  }(0,0)
= z_{1} \int_{0}^{1} \frac{\partial ^{2} F}{\partial z_{2} \partial z_{1}} (tz_{1},0) dt, $$
we obtain
$$\Big| \frac{\partial   F}{\partial z_{2}  }(z_{1},0)\Big|\leq
\Big|\frac{\partial   F}{\partial z_{2}  }(0,0)\Big|+
\frac{3^{\alpha+\frac{1}{2}}\|F\|_{\mathcal{B}^{\alpha}(B^{2})}}
{2^{\alpha}(\alpha-\frac{1}{2})} \big[ (1-|z_{1}| )^{-\alpha+\frac{1}{2}}-1 \big].$$
Hence, for $\alpha>\frac{1}{2}$,
$$(1-|z_{1}| )^{\alpha-\frac{1}{2}}\Big| \frac{\partial   F}{\partial z_{2}  }(z_{1},0)\Big|
\leq\big[1+\frac{3^{\alpha+\frac{1}{2}}}
{2^{\alpha}(\alpha-\frac{1}{2})} \big]\|F\|_{\mathcal{B}^{\alpha}(B^{2})},$$
then (\ref{TM}) follows.
Based on this  result and  applying the method used in Lemma \ref{Ouyang1} (i),   we can easily obtain  that, for $f \in {\mathcal{B}^{\alpha}}$,
$$(1-\|x\|^{2})^{\alpha-\frac{1}{2}}|Df(x)(y)|\leq C(\alpha) \|f\|_{1,\alpha},$$
for   $x\in \mathbb{B}, y\in \partial \mathbb{B}$ with $\langle x,y\rangle=0$.

Notice that,for  any $w=zx+y\in E$ with $z\in \mathbb{C},
x\in \mathbb{B}$ and $y\in   E$   such that  $\langle x,y\rangle = 0$,
$$ \frac{|Df(x)(w)|( 1- \|x\|^{2})^{ \alpha}} {\sqrt{( 1-\|x\|^{2})\|w\|^{2}+|\langle w,x\rangle|^{2}}}
\leq   \|Df(x)\|( 1- \|x\|^{2})^{\alpha}   + \Big|D
(x)(\frac{y}{\|y\|})\Big|( 1- \|x\|^{2})^{\alpha-\frac{1}{2}},$$
which shows   ${\mathcal{B}^{\alpha}} \subseteq T_{\alpha}$ for $\alpha>\frac{1}{2}$, as desired.
 \end{proof}
\begin{proof}[Proof of Theorem \ref{second}]
(i)$\Rightarrow$(ii) Suppose that  $f$ is a Bloch function on $\mathbb{B}$.
Let $D(z_{0},r)$ be a schlicht disk in the range of $f$. Then there exists
a holomorphic function  $g:\mathbb{D} \rightarrow \mathbb{B} $ such  that
$$ f\circ g(z)=z_0+rz.$$
Applying  the Schwarz lemma  for   holomorphic functions  (cf.  \cite[p. 287]{Graham-Kohr}),  we have
$$ \frac{ ( 1- \|g(0)\|^{2})\|Dg(0)\|^{2}+|\langle g(0),Dg(0)\rangle|^{2} }{( 1- \|g(0)\|^{2})^{2}}\leq1.$$
Let $g(0)=x $ and $ Dg(0)=w$. Thus
$$r=|Df(x)(w)|\leq \frac{|D f(x)(w)|( 1- \|x\|^{2})}{\sqrt{( 1- \|x\|^{2})\|w\|^{2}+|\langle w,x\rangle|^{2}}} \leq \|\widetilde{\nabla} f(x)\|,$$
which shows that the radii of the schlicht disks in the range of $f$ are bounded above by $Q_{f}:=\sup_{x\in \mathbb{B}} \|\widetilde{\nabla} f(x)\|$.

(ii)$\Rightarrow$(i) Suppose the radii of the schlicht disks in the
range of $f$ are bounded above by $R$.
For any fixed $y\in \partial\mathbb{B}$,  define $g:\mathbb{D} \rightarrow \mathbb{B} $ by $g(z) = zy$. Fix $x\in \mathbb{B}$. By Bloch's theorem, the holomorphic function
$f\circ \varphi_{x}\circ g$   has a schlicht disk in its range of radius
$$B|(f\circ \varphi_{x}\circ g)'(0)|=B|Df\circ \varphi_{x}(0)(y)|,$$
 where $B$ denotes Bloch's constant.

Therefore, by assumption,   it follows that
$$ |Df\circ \varphi_{x}(0)(y)|\leq \frac{R}{B},$$
for all $x\in \mathbb{B}$ and $y\in \partial\mathbb{B}$, thus $f$ is a Bloch function,  as desired.

(ii)$\Leftrightarrow$(iii)
In the proofs  above, we have
$$   R \leq Q_{f} \leq \frac{R}{B}.$$
Note  that the schlicht disks in the range of
$f$ are exactly those disks which are schlicht disks in the range of $f\circ g$ for some  $g\in  H(\mathbb{D},\mathbb{B})$.
The desired result follows.

(iii)$\Leftrightarrow$(iv)
Following the same  arguments as in (6)$\Leftrightarrow$(7) in \cite[Theorem 3.4]{Timoney1}, one can  prove our result
 and we omit its details here.

(v)$\Rightarrow$(i)
Note that any  $h\in \mathcal{F}_{f}$ satisfies $h(0)=0$. By hypothesis, we have
$$ \big\{ \|Df\circ\phi(0)\|:  \phi\in \text{Aut}(\mathbb{B})\big\}= \big\{ \|Dh(0)\|:  h\in \mathcal{F}_{f}\big\} $$
is bounded, as claimed.

(i)$\Rightarrow$ (v) From the inequality
$$\|Df\circ\phi(x)\|=\|Df\circ\phi\circ \varphi _{x} (0) \big(D\varphi_{x}(0)\big)^{-1}\|\leq \frac{Q_{f}}{1-\|x\|^{2}},$$
it follows that
the family
 $   \mathcal{F}_{f} $
is a family of  normal functions such that   $M_{h}$ is uniformly bounded above by $Q_{f}$. Now the proof  is complete.

\end{proof}

\section{Applications}
Hardy and Littlewood \cite{HL} gave a characterization of
the holomorphic  Lipschitz space $\Lambda_{\alpha}(\mathbb{D})$ of order ${\alpha}\in (0,1]$ on the unit open disk $\mathbb{D}$ of  $\mathbb{C}$, which states that
  a holomorphic  function  $f$   on $\mathbb{D}$ satisfies
  $$  \sup\limits_{\stackrel{z,w\in \mathbb{D}}{z\neq w}}
\frac{|f(z)-f(w)|}{|z-w|^\alpha}<+\infty$$
    if and only if
\begin{eqnarray}\label{Thm:hardy-Littlewood-disc}
\sup_{z\in \mathbb{D}}(1-|z|^{2})^{1-\alpha}|f'(z)| <+\infty. \end{eqnarray}

In \cite{Krantz}, Krantz extended  this result  to  harmonic functions.  See also the  version of the Hardy-Littlewood theorem for quaternionic slice regular functions \cite{Ren-Xu}. As an application of Theorem \ref{main-first}, we first establish the Hardy-Littlewood theorem in the infinite-dimensional Hilbert space.

Denote
$$ {\rm  Lip}\alpha:=\Big\{f\in  H(\mathbb{B}): \sup\limits_{\stackrel{x,y\in \mathbb{B}}{x\neq y}}
\frac{|f(x)-f(y)|}{\|x-y\|^\alpha}<+\infty\Big\}.$$
\begin{theorem}
Let   $\alpha\in (0,1]$. Then   ${\rm  Lip}\alpha =\mathcal{B}^{1-\alpha}$.
\end{theorem}

\begin{proof}
Taking  the same arguments as in \cite[Lemma 6.4.8]{Rudin}, one can show easily the inclusion $\mathcal{B}^{1-\alpha} \subseteq    $ Lip$ \alpha$.

Conversely, let $f \in $Lip$ \alpha$ and $x \in \partial \mathbb{B}$.
 Let us consider the holomorphic function $ F: \mathbb{D}\rightarrow \mathbb{C}$ given by $ F(z)=f(zx)$, which is in  $\Lambda_{\alpha}(\mathbb{D})$.  By the classical Hardy-Littlewood theorem, we have
  $$\sup_{z\in \mathbb{D}}(1-|z|^{2})^{1-\alpha}|F'(z)| <+\infty,$$
 which implies that
 $$\sup_{x\in \mathbb{B}}(1-\|x\|^{2})^{1-\alpha}|\mathcal{R}f(x)| <+\infty.$$
From Theorem \ref{main-first}, it follows that $f\in \mathcal{B}^{1-\alpha}$.
\end{proof}
Holland and Walsh \cite{Holland-Walsh}  further considered  Hardy-Littlewood theorem  to
the limit case $\alpha=0$ for
holomorphic  Bloch spaces and proved that a holomorphic  function  $f$   on $\mathbb{D}$ satisfies
$$\sup_{z\in \mathbb{D}}(1-|z|^{2})|f'(z)| <+\infty $$
if and only if
\begin{eqnarray}\label{Thm:Holland-Wolsh}
\sup\limits_{\stackrel{z,w\in \mathbb{D}}{z\neq w}}\sqrt{(1-|z|^{2})(1-|w|^{2})}
\Big|\frac{f(z)-f(w)}{z-w}\Big|<+\infty.
\end{eqnarray}

The   extensions   to higher dimensions of  the Holland-Walsh result as in (\ref{Thm:Holland-Wolsh}) were obtained  in \cite{Nowak,RenTu} for holomorphic functions in the unit ball of $\mathbb{C}^{n}$. Later, Pavlovi\'c found   that the   Holland-Walsh result holds  even for an arbitrary  $C^1$-function defined on the unit ball of $\mathbb{R}^{n}$ \cite{PHolland}.

In \cite{Zhao} Zhao gave  a  characterization of holomorphic Bloch type spaces on the unit ball of $ \mathbb{C}^{n}$.
\begin{theorem}\label{Zhao} Let $0<\alpha \leq 2$. Let $\lambda$ be any real number satisfying the following properties:
 \begin{enumerate}
\item [(i)]$0 \leq\lambda \leq \alpha$ if $0< \alpha<1 $;
\item [(ii)] $0 <\lambda<1$ if $\alpha=1 $;
\item [(iii)]$\alpha-1\leq\lambda \leq 1$ if $1<\alpha\leq 2 $.
 \end{enumerate}
 Then a holomorphic function $f$ on the open unit ball  $B^{n}$ of $\mathbb{C}^{n}$ is such that
 $$\sup_{z\in {B}^{n}} (1-|z|^{2})^{\alpha}|\nabla f(z)| <+\infty$$
if and only if
$$ \sup_{z,w\in {B}^{n}, z\neq w} (1-|z|^{2})^{ \lambda}(1-|w |^{2})^{\alpha-\lambda} \frac{ |f(z)-f(w)|}{ |z-w|}<+\infty. $$
\end{theorem}

For more relative equivalent characterizations of Bloch type functions in the finite-dimensional Euclidean space, we refer to \cite{Chen-HuaiHui,Li-Wulan,Zhang-Chen,Chen} and references therein.

In \cite{Zhao}, Zhao also offered    some examples to show that the
conditions on $\alpha$ and $\lambda$ in Theorem \ref{Zhao}  cannot be improved. In fact, Theorem \ref{Zhao} does  hold for $\alpha=\lambda=0$ and hold for any $C^{1}$-function  defined on the unit ball of any infinite dimensional Hilbert space.  Recently, Dai and Wang in \cite{Dai} reveal the reason in theory that some equivalent characterizations of the Bloch type space  require extra conditions for $\alpha$.  In present paper, we shall show the main results in \cite{Dai} still hold for holomorphic functions defined in
  any infinite-dimensional Hilbert space.

Denote
 $$ S_{\alpha,\lambda}:=\Big\{ f\in H(\mathbb{B}):  \sup_{x,y\in \mathbb{B} , x\neq y} (1-\|x\|^{2})^{ \lambda}
 (1-\|y\|^{2})^{\alpha-\lambda}
\frac{ |f(x)-f(y) |}{ \|x-y\|}<+\infty\Big \}. $$
 \begin{theorem}\label{Dai}
\begin{enumerate}
\item  [(i)]
 Let $\alpha, \lambda$ be any real numbers satisfying the following properties:
\begin{enumerate}
\item  $0<\lambda< \alpha-1$ if $1 <\alpha\leq 2$;
\item  $0<\lambda \leq \alpha/2$ if $ \alpha>2$.
\end{enumerate}
Then $S_{\alpha,\lambda}=\mathcal{B}^{\lambda+1}$.
\item [(ii)] Let $\alpha, \lambda$ be any real numbers satisfying the following properties:
\begin{enumerate}
\item  $1<\lambda< \alpha $ if $1<\alpha\leq 2$;
\item  $\alpha/2 <\lambda\leq \alpha $ if $ \alpha>2$.
\end{enumerate}
Then $S_{\alpha,\lambda}= \mathcal{B}^{\alpha-\lambda+1}$.
 \item [(iii)] Let   $\alpha\geq1$.  Then $S_{\alpha,\lambda}= H^{\infty}$  for $\lambda=0$ or $\lambda=\alpha$.
 \end{enumerate}
\end{theorem}

From Theorems \ref{Zhao} and \ref{Dai}, the space $S_{\alpha,\lambda}$ is described completely for  all  cases $0\leq\lambda\leq\alpha$.

In order to prove Theorem \ref{Dai}, we first  generalise  a result in  \cite{Zhu1} by Theorem \ref{main-first} as follows.
\begin{lemma}\label{Zhu}
 Let $\alpha>1$.  Then $\mathcal{B}^{\alpha}=H_{\alpha-1}^{\infty}$.
\end{lemma}

\begin{proof}
Let  $x \in \partial \mathbb{B}$ and $f\in H_{\alpha-1}^{\infty}$ with $\|f\|_{H_{\alpha-1}^{\infty}}=1$.
Let us consider the holomorphic function $ F: \mathbb{D}\rightarrow \mathbb{C}$ given by $F(z)=f(zx)$.
 Then $( 1- |z|^{2} )^{ \alpha-1}|F(z)|\leq1$.
By \cite[Proposition 7]{Zhu1}, there exists some constant $C>0$ such that
$$ (1-|z|^{2})^{\alpha}|F'(z)|\leq C, \ \forall z\in \mathbb{D},$$
that is
$$( 1-|z|^{2})^{\alpha}|Df(zx)(x)|\leq C, \ \forall z\in \mathbb{D},$$
which implies
$$( 1-\|y\|^{2})^{\alpha}|\mathcal{R} f(y)|\leq C, \ \forall y\in \mathbb{B}.$$
Applying  Theorem \ref{main-first}, if follows that $f\in \mathcal{B}^{\alpha} $.

Conversely, let  $f\in \mathcal{B}^{\alpha} $, thus we have  $F\in \mathcal{B}^{\alpha}$,
which is also    in $H_{\alpha-1}^{\infty}$ by \cite[Proposition 7]{Zhu1} again. Consequently, we have $f\in H_{\alpha-1}^{\infty}$.
 \end{proof}

\begin{proof}[Proof of Theorem \ref{Dai}]
(i) Let  $f\in \mathcal{B}^{\lambda+1}$. For $x, y \in \mathbb{B}$,
we choose a path $\gamma(t)=tx+(1-t)y, t\in [0,1]$ connecting $x$ and $y$. Then  it follows that
\begin{eqnarray*}\label{norm}
 |f(x)-f(y)|
  &=& \Big|\int_{0}^{1}\frac{d}{dt}f(\gamma(t))dt \Big|
  \\
  &=& \Big|\int_{0}^{1}Df(\gamma(t))(x-y)dt \Big|
  \\
 &\leq &\int_{0}^{1}\Big|Df(\gamma(t))(x-y)\Big| dt
 \\
&\leq &  \int_{0}^{1}\|D f(\gamma(t))\| \|x-y\| dt
  \\
  &\leq & C \|x-y\| \int_{0}^{1} \frac{1}
  {(1-\|\gamma(t) \|^{2})^{\lambda+1}}dt
   \\
   &\leq & C \|x-y\| \int_{0}^{1} \frac{1}
   {(1-t\|x\|-(1-t)\|y\| )^{\lambda+1}}dt.   \end{eqnarray*}
From the proof of  \cite[Theorem 3.1]{Dai}, we have
$$ \int_{0}^{1} \frac{1} {(1-t\|x\|-(1-t)\|y\|)^{\lambda+1}}dt \leq
 \frac{ C}{(1- \|x\|^{2})^{ \lambda}(1- \|y\|^{2})^{\alpha-\lambda}}.$$
Hence,
$$  |f(x)-f(y)|   \leq \frac{ C  \|x-y\| }
 {(1- \|x\|^{2})^{ \lambda}(1- \|y\|^{2})^{\alpha-\lambda}},$$                    which shows that $f \in S_{\alpha,\lambda} $.

Conversely,  the method     in \cite[Theorem 3.1]{Dai}
can applied  word by word to   prove  the inclusion
$S_{\alpha,\lambda}   \subseteq\mathcal{B}^{\lambda+1}$
by Lemma \ref{Zhu}.

(ii) follows by using (i).
It is easy to check (iii) if we can prove that
\begin{eqnarray}  \label{inequa}
|f(x)-f(a) |\leq    2\frac{\|x-a\|}{ 1-\|x\|}, \ \  \forall x,a\in \mathbb{B},
\end{eqnarray}
for holomorphic functions $f\in H^{\infty}$ with $\|f\|_{H^{\infty}}=1$.

Let us  show   inequality (\ref{inequa}).  From the Schwarz lemma for holomorphic functions, we have
$$ |f(x)-f(0) |\leq  2\|x\|, \ \  \forall x\in \mathbb{B}.$$
Applying this inequality to the holomorphic function $f\circ\varphi_{a}$, it holds that
$$ |f(x)-f(a) |\leq  2\| \varphi_{a}(x)\|\leq  2\frac{\|a-x\|}{|1-\langle x,a\rangle|}\leq  2\frac{\|x-a\|}{ 1-\|x\|},\ \  \forall x,a\in \mathbb{B},$$
as desired.
\end{proof}

$\bold{Acknowledgemants}$ \
The main result of this work is part of the author¡¯s Ph.D. thesis and
completed in November 2016. The author is very grateful to his advisor  for helpful discussions
and to the anonymous referees for valuable suggestions for the improvement of the manuscript.
\bibliographystyle{amsplain}

\vskip 10mm

\end{document}